\documentclass[leqno,12pt]{article}
\usepackage{latexsym}
\usepackage{amsmath}
\usepackage{amssymb}
\usepackage{enumerate}

\usepackage{theorem}
\newtheorem{theorem}{Theorem}[section]

\newtheorem{lemma}[theorem]{Lemma}
\newtheorem{corollary}[theorem]{Corollary}

\newtheorem{thmA}{Theorem A}
\newtheorem{thmB}{Theorem B}

\theorembodyfont{\rmfamily}
\newtheorem{proof}{\textmd{\textit{Proof.}}}

\newtheorem{remark}[theorem]{Remark}
\newtheorem{example}[theorem]{Example}
\newtheorem{definition}[theorem]{Definition}

\makeatletter

\@addtoreset{equation}{section}
\makeatother

\newcommand{\qedd}{\hfill \Box}

\newcommand{\lra}{\longrightarrow}
\newcommand{\ora}{\overrightarrow}
\newcommand{\ola}{\overleftarrow}
\newcommand{\wt}{\widetilde}
\newcommand{\wh}{\widehat}
\newcommand{\ol}{\overline}
\newcommand{\N}{\ensuremath{\mathbb{N}}}
\newcommand{\R}{\ensuremath{\mathbb{R}}}
\newcommand{\Sph}{\ensuremath{\mathbb{S}}}

\newcommand{\cG}{\ensuremath{\mathcal{G}}}

\newcommand{\cN}{\ensuremath{\mathcal{N}}}
\newcommand{\cT}{\ensuremath{\mathcal{T}}}
\newcommand{\cO}{\ensuremath{\mathcal{O}}}

\def\diam{\mathop{\mathrm{diam}}\nolimits}

\def\Cut{\mathop{\mathrm{Cut}}\nolimits}

\title{Topology of complete Finsler manifolds\\
with radial flag curvature bounded below\footnote{
2010 Mathematics Subject Classification: Primary 53C60; Secondary 53C21, 53C22.}
\footnote{Key words and phrases:
Toponogov's comparison theorem, Finsler geometry, flag curvature, finiteness theorem}}

\author{Kei KONDO \ $\cdot$ \ Shin-ichi OHTA\footnote{
Supported in part by Grant-in-Aid for Young Scientists (B) 23740048}
\ $\cdot$ \ Minoru TANAKA}
\date{\today}
\begin{document}
\maketitle

\begin{abstract}
We recently established a Toponogov type triangle comparison theorem
for a certain class of Finsler manifolds whose radial flag curvatures are bounded 
below by that of a von Mangoldt surface of revolution.
In this article, as its applications, we prove the finiteness of topological type and 
a diffeomorphism theorem to Euclidean spaces.
\end{abstract}

\section{Introduction}\label{sec1}

This article is a continuation of \cite{KOT}.
In \cite{KOT}, we have established a Toponogov type triangle comparison theorem (TCT) 
for a certain class of Finsler manifolds whose radial flag curvatures are bounded 
below by that of a von Mangoldt surface of revolution
(see Theorem~\ref{TCT} for the precise statement).
In this article, we prove several applications of our Toponogov theorem
on the relationship between the topology and the curvature of a Finsler manifold.
We remark that, compared to the Riemannian case,
there are only a small number of such kind of results,
e.g., Rademacher's quarter pinched sphere theorem (\cite{Ra}),
Shen's finiteness theorem under lower Ricci and mean (or ${\mathbf S}$-) curvature bounds (\cite{Shvol}), the second author's generalized splitting theorems
under nonnegative weighted Ricci curvature
(\cite{Ospl}), 
and the first author's generalized diameter sphere theorem 
with radial flag curvature bounded from below by $1$ as an application of TCTs (\cite{K}).

\medskip

In order to state our results, let us introduce several notions
in Finsler geometry as well as the geometry of radial curvature.
Let $(M, F, p)$ denote a pair of a forward complete, connected,
$n$-dimensional $C^\infty$-Finsler manifold $(M,F)$ with a base point $p \in M$,
and $d: M \times M \lra [0, \infty)$ denote the distance function induced from $F$.
We remark that the {\em reversibility} $F(-v)=F(v)$ is not assumed in general,
so that $d(x, y) \not= d(y,x)$ is allowed.

For a local coordinate $(x^i)^{n}_{i=1}$ of an open subset $\cO \subset M$,
let $(x^i, v^j)_{i,j=1}^{n}$ be the coordinate of the tangent bundle $T\cO$ over $\cO$ such that 
\[ v:= \sum_{j = 1}^{n} v^j \frac{\partial}{\partial x^j}\Big|_{x},
 \qquad \ x \in \cO. \]
For each $v \in T_xM \setminus \{0\}$, the positive-definite $n \times n$ matrix 
\[
\big( g_{ij} (v) \big)_{i,j= 1}^{n}:= 
\left(
\frac{1}{2} \frac{\partial^2 (F^2)}{\partial v^i \partial v^j}(v)
\right)_{i, j = 1}^{n}
\]
provides us the Riemannian structure $g_{v}$ of $T_x M$ by 
\[
g_{v} \left(
\sum_{i = 1}^{n} a^i \frac{\partial}{\partial x^i}\bigg|_{x}, 
\sum_{j = 1}^{n} b^j \frac{\partial}{\partial x^j}\bigg|_{x}
\right) 
:=
\sum_{i,j = 1}^{n} g_{ij} (v) a^ib^j. 
\]
This is a Riemannian approximation (up to the second order)
of $F$ in the direction $v$.
For two linearly independent vectors $v, w \in T_{x} M \setminus \{0\}$,
the {\em flag curvature} is defined by 
\[
K_M (v, w) := \frac{g_{v} (R^{v} (w, v)v, w)}{g_{v} (v, v) g_{v} (w, w) - g_{v} (v, w)^2},
\]
where $R^{v}$ denotes the curvature tensor induced from the Chern connection
(see \cite[\S 3.9]{BCS} for details).
We remark that $K_M (v, w)$ depends not only on the \emph{flag} $\{sv + tw\,|\, s, t \in \R\}$,
but also on the \emph{flag pole} $\{sv\,|\, s> 0\}$.

Given $v,w \in T_xM \setminus \{0\}$, define the \emph{tangent curvature} by 
\[
\cT_M(v, w) := g_X\big( D^Y_Y Y(x) - D^X_Y Y(x), X(x) \big), 
\]
where the vector fields $X,Y$ are extensions of $v,w$,
and $D_{v}^{w}X(x)$ denotes the covariant derivative of $X$ by $v$ with reference vector $w$. 
Independence of $\cT_M(v,w)$ from the choices of $X,Y$ is easily checked.
Note that $\cT_M \equiv 0$ if and only if $M$ is of \emph{Berwald type}
(see \cite[Propositions~7.2.2, 10.1.1]{Sh}).
In Berwald spaces, for any $x,y \in M$, 
the tangent spaces $(T_xM, F|_{T_xM})$ and $(T_yM, F|_{T_yM})$
are mutually linearly isometric (cf.~\cite[Chapter~10]{BCS}).
In this sense, $\cT_M$ measures the variety of tangent Minkowski normed spaces.
\medskip

Let $\wt{M}$ be a complete $2$-dimensional Riemannian manifold,
which is homeomorphic to $\R^{2}$. 
Fix a base point $\tilde{p} \in \wt{M}$.
Then, we call the pair $(\wt{M}, \tilde{p})$ a {\em model surface of revolution} 
if its Riemannian metric $d\tilde{s}^2$ is expressed 
in terms of the geodesic polar coordinate around $\tilde{p}$ as 
\[
d\tilde{s}^2 = dt^2 + f(t)^2d \theta^2, \qquad 
(t,\theta) \in (0,\infty) \times \Sph_{\tilde{p}}^1, 
\]
where $f : (0, \infty) \lra \R$ is a positive smooth function 
which is extensible to a smooth odd function around $0$, 
and $\Sph^{1}_{\tilde{p}} := \{ v \in T_{\tilde{p}} \wt{M} \,|\, \| v \| = 1 \}$. 
Define the {\em radial curvature function} $G: [0,\infty) \lra \R$
such that $G(t)$ is the Gaussian curvature at $\tilde{\gamma}(t)$,
where $\tilde{\gamma}:[0,\infty) \lra \wt{M}$ is any (unit speed) meridian
emanating from $\tilde{p}$.
Note that $f$ satisfies the differential equation 
$f''+Gf=0$ with initial conditions $f(0) = 0$ and $f'(0) = 1$. 
We call $(\wt{M}, \tilde{p})$ a {\em von Mangoldt surface}
if $G$ is non-increasing on $[0,\infty)$. 
Paraboloids and $2$-sheeted hyperboloids are typical examples of 
von Mangoldt surfaces. An atypical example of such a surface is the following.

\begin{example}{\rm (\cite[Example 1.2]{KT1})}
Set $f (t) := e^{- t^{2}} \tanh t$ on $[0, \infty)$.
Then the non-compact surface of revolution 
$(\wt{M}, \tilde{p})$ with $d\tilde{s}^2 = dt^2 +  f(t)^2d \theta^2$ is of von Mangoldt type, 
and $G$ changes the sign.
Indeed, $\lim_{t \downarrow 0}G (t) = 8$ and $\lim_{t \to \infty}G (t) = - \infty$.
\end{example}

We say that a Finsler manifold $(M, F, p)$ has the
{\em radial flag curvature bounded below by that of a model surface of revolution 
$(\wt{M}, \tilde{p})$} if, 
along every unit speed minimal geodesic $\gamma: [0,l) \lra M$ 
emanating from $p$, we have
\[
K_{M} \big(\dot{\gamma}(t), w \big) \ge G (t)
\]
for all $t \in [0, l)$ and $w \in T_{\gamma(t)}M$ 
linearly independent to $\dot{\gamma}(t)$.

\bigskip

We set 
\begin{equation}\label{G_p}
\mathcal{G}_p (x) := \{ \dot{\gamma}(l) \in T_{x}M \,|\,
 \text{$\gamma$ is a minimal geodesic segment from $p$ to $x$} \},
\end{equation}
where $\gamma:[0,l] \lra M$ with $l=d(p,x)$, and
\[ B_r^+ (p):=\{ x \in M \,|\, d(p,x)<r \}, \qquad
 \diam\left( \partial B_r^+ (p) \right):=\sup_{q_1,\,q_2 \in \partial B_r^+ (p)}d(q_1,q_2). \]
Then, our first main result is a finiteness theorem of topological type.

\begin{thmA}\label{2012_03_26_main1}
Let $(M, F, p)$ be a forward complete, non-compact, connected $C^{\infty}$-Finsler manifold
whose radial flag curvature is bounded below by that of a von Mangoldt surface 
$(\wt{M}, \tilde{p})$ satisfying $f'(\rho) =0$ and $G(\rho) \ne0$ 
for a unique $\rho \in (0, \infty)$. Assume that, for some $t_0 > \rho$, 
\begin{enumerate}[{\rm (1)}]
\item 
$\diam(\partial B_t^+ (p)) = O (t^\alpha)$ for some $\alpha \in (0, 1)$ as $t \to \infty$,
\item 
$g_v(w, w) \ge F(w)^2$ for all $x \in M \setminus \ol{B_{t_0}^+ (p)}$,
$v \in \cG_p (x)$ and $w  \in T_xM$, 
\item
$\cT_{M} (v, w) = 0$ for all $x \in M \setminus \ol{B_{t_0}^+ (p)}$,
$v \in \cG_p (x)$ and $w  \in T_xM$, 
\item
the reverse curve $\bar{c}(s):=c(l-s)$ of any minimal geodesic segment
$c:[0,l] \lra M \setminus \ol{B_{t_0}^+ (p)}$ is geodesic. 
\end{enumerate}
Then $M$ has finite topological type, i.e., $M$ is homeomorphic to the interior of a compact manifold 
with boundary.
\end{thmA}

\begin{remark}
All conditions in Theorem A are sufficient ones that make our TCT 
hold (see Theorem~\ref{TCT}): The condition (1) guarantees the condition (1) 
in Theorem~\ref{TCT}. The biggest obstruction when we establish a TCT in Finsler geometry 
is the covariant derivative even though $F$ is reversible. 
By the condition (2) and $f' <0$ on $(\rho, \infty)$ (because of $f'(\rho) =0$ and $G(\rho) \ne0$), we can overcome the obstruction, i.e., 
thanks to the (2), we can transplant the strictly concaveness of 
$\wt{M} \setminus \ol{B_{t_0} (\tilde{p})}$ to $M \setminus \ol{B_{t_0}^+ (p)}$
 (see \cite[Section 3]{KOT} for more details), where the convexity on $\wt{M} \setminus \ol{B_{t_0} (\tilde{p})}$ arises from the negative second fundamental form for $f' <0$ on $(\rho, \infty)$. 
 Note that the (2) is the \emph{$2$-uniform convexity} with the sharp constant (see \cite{Ouni}), but, in our situation, {\bf only for special} points and directions. 
 This means that the convexity holds only along all minimal geodesic 
segments emanating from $p$ in our theorem. 
It is very natural thing to assume that the condition (3), if we employ a {\bf Riemannian} 
model surface of revolution $\wt{M}$ as a reference surface. 
Here note that $\cT_{\wt{M}} \equiv 0$. 
It is {\bf not difficult} to construct non-Riemannian spaces satisfying (2) and (3) 
(see Example \ref{exa2013_09_03}).
\end{remark}

\begin{example}(\cite{KOT})\label{exa2013_09_03}\\[1mm]
$\bullet$ Let $(M,g, p)$ be a complete non-compact Riemannian manifold whose 
radial (sectional) curvature is bounded below by that of a von Mangoldt surface 
$(\wt{M}, \tilde{p})$ satisfying $f'(\rho) =0$ and $G(\rho) \ne0$ for unique $\rho \in (0, \infty)$. 
Modify (the unit spheres of) $g$ on $M \setminus B^+_{\rho}(p)$,
outside a neighborhood of $\bigcup_{z \in M \setminus B^+_{\rho}(p)} \cG_p(z)$,
in such a way that the (2) holds. Note that the resulting non-Riemannian metric still satisfies 
the (3), because this modification does not affect $g_v$ for $v \in \bigcup_{z \in M \setminus B^+_{\rho}(p)} \cG_p(z)$.\\[2mm]
$\bullet$ Let $(M, F, p)$ be the Finsler manifold satisfying the radial flag curvature conditions on Theorem A. If $F$ is Riemann on $M \setminus \ol{B_\rho^+(p)}$, 
then $(M, F, p)$ satisfies all conditions in Theorem A except for the (1). E.g.,  
\[
F(v) = 
\begin{cases}
\ \sqrt{g(v, v)} + \beta (v)  \ &\text{on} \ B_\rho^+(p)\\[4mm]
\ \sqrt{g(v, v)}  \ &\text{on} \ M \setminus \ol{B_\rho^+(p)}
\end{cases}
\]
etc.
\end{example}

By changing the structure of $F$, we can reduce a few conditions in Theorem A:

\begin{corollary}
Let $(M, F, p)$ be a forward complete, non-compact, connected $C^{\infty}$-Finsler manifold
whose radial flag curvature is bounded below by that of a von Mangoldt surface 
$(\wt{M}, \tilde{p})$ satisfying $f'(\rho) =0$ and $G(\rho) \ne0$ for unique $\rho \in (0, \infty)$. 
Assume that, for some $t_0 > \rho$, 
\begin{enumerate}[{\rm (1)}]
\item 
$\diam(\partial B_t^+ (p)) = O (t^\alpha)$ for some $\alpha \in (0, 1)$ as $t \to \infty$,
\item 
$g_v(w, w) \ge F(w)^2$ for all $x \in M \setminus \ol{B_{t_0}^+ (p)}$,
$v \in \cG_p (x)$ and $w  \in T_xM$, 
\end{enumerate}
If $F$ is of {\bf Berwald} type on $M \setminus \ol{B_{t_0}^+ (p)}$, 
then $M$ has finite topological type.
\end{corollary}

\begin{remark}The condition (4) in Theorem A always holds, 
if $F$ is reversible on $M \setminus \ol{B_{t_0}^+ (p)}$. 
In the case where $F$ is Riemannian, 
the diameter growth bound (1) seems to be very restrictive. 
Indeed, if we employ a non-negatively curved non-compact model surface of revolution
$(\wt{M}, \tilde{p})$ having the diameter growth $o(t^{1/ 2})$,
then $M$ is isometric to the $n$-dimensional model space $\wt{M}^n$
(see \cite[Theorem~1.2]{ST}, \cite[Example~1.1]{KT2}). 
Hence, if $F$ is Riemannian, then we can prove, without the growth condition, 
the finiteness of topological type of a complete non-compact Riemannian manifold with radial curvature bounded below by that of an {\bf arbitrary} non-compact model surface 
of revolution admitting a finite total curvature(see \cite[Theorem~2.2]{KT2} and \cite[Theorem~1.3]{TK}).
\end{remark}

By an entirely different technique, if $F$ is reversible, then we can improve Theorem~A  as follows:

\begin{thmB}
Let $(M, F, p)$ be a forward complete, non-compact, connected $C^{\infty}$-Finsler manifold
whose radial flag curvature is bounded below by that of a von Mangoldt surface 
$(\wt{M}, \tilde{p})$ satisfying $f'(\rho) =0$ and $G(\rho) \ne0$ for unique $\rho \in (0, \infty)$. 
Assume that, for some $t_0 > \rho$, 
\begin{enumerate}[{\rm (1)}]
\item 
$\diam(\partial B_t^+ (p)) = O (t^\alpha)$ for some $\alpha \in (0, 1)$ as $t \to \infty$,
\item 
$g_v(w, w) \ge F(w)^2$ for all $x \in M \setminus \ol{B_{t_0}^+ (p)}$,
$v \in \cG_p (x)$ and $w  \in T_xM$, 
\item
$\cT_{M} (v, w) = 0$ for all $x \in M \setminus \ol{B_{t_0}^+ (p)}$,
$v \in \cG_p (x)$ and $w  \in T_xM$. 
\end{enumerate}
If $F$ is reversible, then $M$ is diffeomorphic to $\R^n$ and,
for every unit speed minimal geodesic $\gamma: [0,\infty) \lra M$ emanating from $p$,
we have $K_{M}(\dot{\gamma}(t), w) = G (t)$ for all $t>0$.
\end{thmB}

\begin{remark}
In Theorem B, we can remove the condition (3), if we additionally 
assume that $M$ is of Berwald type. 
The result related to Theorem B is Shiohama 
and the third author's \cite[Theorem~1.2]{ST}, 
where they proved that a complete non-compact Riemannian manifold is isometric to 
the $n$-dimensional model space $\wt{M}^n$
if its radial curvature is bounded below by that of a non-compact model surface of revolution
$\wt{M}$ satisfying $\int_1^{\infty} f(t)^{-2} \,dt = \infty$.
Observe that our von Mangoldt surface always satisfies this integration assumption. 
However, in our Finsler situation, 
it is difficult (and in fact impossible in many cases) to obtain isometry to a model space. 
That is, spaces of constant flag curvatures are not unique. E.g., 
all Minkowski normed spaces have the flat flag curvature and 
all Hilbert geometries satisfy $K_M \equiv -1$ (cf.\ \cite{Shspr}). 
Other result related to Theorem~B is the first and the third authors' \cite[Theorem~1.1]{KT3} 
on a complete non-compact connected Riemannian manifold with smooth convex boundary.
\end{remark}

\section{A Toponogov type triangle comparison theorem}\label{sec2}

We first recall the Toponogov type triangle comparison theorem
established in \cite[Theorem~1.2]{KOT}.
We refer to \cite{BCS} and \cite{Sh} for the basics of Finsler geometry.

Let $(M,F,p)$ be a forward complete, connected
$C^\infty$-Finsler manifold with a base point $p \in M$,
and denote by $d$ its distance function.
The forward completeness guarantees that any two points in $M$ can be joined by
a minimal geodesic segment (by the Hopf-Rinow theorem, \cite[Theorem~6.6.1]{BCS}).
Since $d(x,y) \neq d(y,x)$ in general, we also introduce
\[
d_{\rm{m}} (x, y) := \max\{d(x, y), d(y, x)\}.
\]
It is clear that $d_{\rm{m}}$ is a distance function of $M$.
We can define the `angles' with respect to $d_{\rm{m}}$ as follows.

\begin{definition}{\bf (Angles)}\label{2012_03_24_def2.3}
Let $c :[0,a] \lra M$ be a unit speed minimal geodesic segment
(i.e., $F(\dot{c}) \equiv 1$) with $p \not\in c([0,a])$. 
The \emph{forward} and the \emph{backward angles}
$\ora{\angle}(pc(s)c(a))$, $\ola{\angle}(pc(s)c(0)) \in [0,\pi]$ at $c(s)$ are defined via
\begin{align*}
\cos \ora{\angle}\big( pc(s)c(a) \big) &:= -\lim_{h \downarrow 0} 
 \frac{d(p,c(s+h)) -d(p, c(s))}{d_{\rm{m}} (c(s), c(s + h))} \quad \text{for $s \in [0,a)$}, \\
\cos \ola{\angle}\big( pc(s)c(0) \big) &:= \lim_{h \downarrow 0} 
 \frac{d(p, c(s)) - d (p, c(s-h))}{d_{\rm{m}} (c(s -h), c(s))} \quad \text{for $s \in (0, a]$}.
\end{align*}
(These limits indeed exist in $[-1,1]$ thanks to the definition of $d_{\rm m}$,
see \cite[Lemma~2.2]{KOT}).
\end{definition}

\begin{definition}{\bf (Forward triangles)}\label{def2.2_ft}
For three distinct points $p, x, y \in M$,
\[ \triangle (\ora{px}, \ora{py}) := (p, x, y; \gamma, \sigma, c) \]
will denote the \emph{forward triangle} consisting of unit speed minimal geodesic segments
$\gamma$ emanating from $p$ to $x$, $\sigma$ from $p$ to $y$, and $c$ from $x$ to $y$.
Then the corresponding \emph{interior angles} $\ora{\angle}x, \ola{\angle}y$
at the vertices $x$, $y$ are defined by 
\[
\ora{\angle}x := \ora{\angle}\big( p c(0) c(d(x,y)) \big), \qquad
 \ola{\angle}y := \ola{\angle}\big( p c(d(x,y)) c(0) \big).
\]
\end{definition}

\begin{definition}{\bf (Comparison triangles)}
Fix a model surface of revolution $(\wt{M}, \tilde{p})$. 
Given a forward triangle $\triangle (\ora{px}, \ora{py})= (p, x, y; \gamma, \sigma, c)\subset M$,
a geodesic triangle $\triangle (\tilde{p}\tilde{x} \tilde{y}) \subset \wt{M}$ is called
its \emph{comparison triangle} if
\[ \tilde{d}(\tilde{p}, \tilde{x}) = d(p, x), \qquad 
\tilde{d}(\tilde{p},\tilde{y}) = d(p, y), \qquad 
\tilde{d}(\tilde{x},\tilde{y}) = L_{\rm{m}}(c) \]
hold, where we set
\[
L_{\rm{m}}(c):= \int^{d(x,\,y)}_0 \max\{ F(\dot{c}),F(-\dot{c}) \} \,ds.
\]
\end{definition}

Now, the main result of \cite{KOT} asserts the following.

\begin{theorem}[TCT, \cite{KOT}]\label{TCT}
Assume that $(M, F, p)$ is a forward complete, connected $C^{\infty}$-Finsler manifold
whose radial flag curvature is bounded below by that of a von Mangoldt surface 
$(\wt{M}, \tilde{p})$ satisfying $f'(\rho) =0$ and $G(\rho) \ne0$ for a unique $\rho \in (0, \infty)$. 
Let $\triangle (\ora{px}, \ora{py}) = (p, x, y; \gamma, \sigma, c) \subset M$ be a forward triangle 
satisfying that, for some open neighborhood $\cN(c)$ of $c$,
\begin{enumerate}[$(1)$]
\item
$c([0,d(x,y)]) \subset M \setminus \ol{B^+_{\rho} (p)}$,
\item 
$g_v(w,w) \ge F(w)^2$ for all $z \in \cN(c)$, $v \in \cG_p(z)$ and $w \in T_zM$,
\item
$\cT_{M}(v,w) = 0$ for all $z \in \cN(c)$, $v \in \cG_p(z)$ and $w \in T_zM$,
and the reverse curve $\bar{c}(s):=c(d(x,y)-s)$ of $c$ is also geodesic.
\end{enumerate}
If such $\triangle (\ora{px}, \ora{py})$ admits a comparison triangle $\triangle (\tilde{p}\tilde{x} \tilde{y})$ in $\wt{M}$,
then we have $\ora{\angle} x \ge \angle \tilde{x}$ and $\ola{\angle} y \ge \angle \tilde{y}$.
\end{theorem}

\begin{remark}
If a von Mangoldt surface $(\wt{M}, \tilde{p})$ satisfies $G(\rho) = 0$ for a unique $\rho \in (0, \infty)$, then $f'(\rho) =0$ and $f'(t) >0$ on $(\rho, \infty)$. In this case, Theorem \ref{TCT} holds, 
if $F(w)^2 \ge g_v(w,w)$ for all $z \in \cN(c)$, $v \in \cG_p(z)$ and $w \in T_zM$ as in (2). For this, 
see \cite[Remark 2.10]{K}

\end{remark}

\section{Fundamental tools on model surfaces}\label{sec2.5}

We next introduce some fundamental tools in the geometry of model surfaces of revolution.
We refer to \cite[Chapter 7]{SST} for more details.
Let $(\wt{M}, \tilde{p})$ be a non-compact model surface of revolution with its metric
$d\tilde{s}^2 = dt^2 + f(t)^2d \theta^2$ on $(0,a) \times \Sph_{\tilde{p}}^1$.
Given a unit speed geodesic $\tilde{c} : [0, a) \lra \wt{M}$ $(0 < a \le \infty)$
expressed as $\tilde{c}(s) = (t(s), \theta(s))$,
there exists a non-negative constant $\nu$ such that
\begin{equation}\label{2012_03_28_clairaut}
\nu = f\big( t(s) \big)^2 |\theta' (s)|
 = f\big( t(s) \big) \sin \angle\big( \dot{\tilde{c}}(s), (\partial/\partial t)|_{\tilde{c}(s)} \big)
\end{equation}
for all $s \in [0,a)$.
The equation (\ref{2012_03_28_clairaut}) is called the {\em Clairaut relation},
and  $\nu$ is called the {\em Clairaut constant} of $\tilde{c}$.
Note that $\nu=0$ if and only if $\tilde{c}$ is (a part of) a meridian.
Since $\tilde{c}$ has unit speed, we deduce from $|t'|^2+|f(t)\theta'|^2=1$ that
\[
|t'(s)| = \frac{\sqrt{f(t(s))^{2} - \nu^2}}{f(t(s))}.
\]
Thus we observe that $t'(s) = 0$ if and only if $f(t(s)) = \nu$.
Moreover, if $a<\infty$, then the length $L(\tilde{c})$ of $\tilde{c}$
is not less than
\begin{equation}\label{length-ineq_c}
t(a) - t(0) + \frac{\nu^{2}}{2} \int_{t(0)}^{t(a)} \frac{1}{f(t)\sqrt{f(t)^{2} -\nu^{2}}} \,dt.
\end{equation}
The proof of \eqref{length-ineq_c} can be found in (the proof of) \cite[Lemma~2.1]{ST}.

\section{Proof of Theorem~A}\label{sec3}

Let $(M,F,p)$, $f$ and $\rho$ be as in Theorem~A. 
The following fact on the cut loci of a von Mangoldt surface is important.

\begin{remark}\label{2012_09_19_rem4.1}
The cut locus $\Cut (\tilde{x})$ of $\tilde{x} \not= \tilde{p}$ is either an empty set, 
or a ray properly contained in the meridian $\theta^{-1} (\theta (\tilde{x}) + \pi)$ opposite to $\tilde{x}$.
Moreover, the endpoint of $\Cut(\tilde{x})$ is the first conjugate point to $\tilde{x}$ 
along the minimal geodesic from $\tilde{x}$ passing through $\tilde{p}$ (\cite[Main Theorem]{T}).
\end{remark}

We first show an auxiliary lemma on the model surface.

\begin{lemma}\label{2012_03_26_lem3.1}
If two distinct points
$\tilde{x},\tilde{y} \in \wt{M} \setminus \ol{B_{\rho}(\tilde{p})}$ 
satisfy $\tilde{d}(\tilde{p}, \tilde{x}) \le \tilde{d}(\tilde{p}, \tilde{y})$, then
\[
\angle \big( \dot{\tilde{c}} (0), (\partial / \partial t)|_{\tilde{x}} \big) <\pi /2
\]
holds for any unit speed minimal geodesic segment $\tilde{c}$ emanating from $\tilde{x}$ to $\tilde{y}$. 
In particular, we have $\tilde{c}([0, d(\tilde{x}, \tilde{y})]) \subset \wt{M} \setminus \ol{B_\rho (\tilde{p})}$.
\end{lemma}

\begin{proof}
Let us write $\tilde{c}(s) = (t(s), \theta (s))$.
Suppose that $\angle (\dot{\tilde{c}} (0), (\partial / \partial t)|_{\tilde{x}}) \ge \pi/2$
which is equivalent to $t'(0) \le 0$.
Since $f' < 0$ on $(\rho, \infty)$ because of $f'(\rho) =0$ and $G(\rho) \ne0$ for a unique 
$\rho \in (0, \infty)$, it follows from \cite[(7.1.15)]{SST} that 
\[
t''(0) = f\big( t(0) \big) f'\big( t(0) \big) \theta'\big( t(0) \big)^2 <0.
\] 
Hence $t(s)$ is decreasing on $[0, \delta]$ for some small $\delta > 0$. 
Since $t(d(\tilde{x}, \tilde{y})) = \tilde{d}(\tilde{p}, \tilde{y}) \ge \tilde{d}(\tilde{p}, \tilde{x}) = t(0)$, 
there exists $s_0 \in (0, \tilde{d}(\tilde{p}, \tilde{y}))$ such that $t'(s_0) = 0$ and $t(s_0)<t(0)$.
By the Clairaut relation \eqref{2012_03_28_clairaut}, for any $s \in [0, \tilde{d}(\tilde{p}, \tilde{y})]$,
we observe
\[
f \big( t(s_0) \big) =
f\big( t(s) \big) \sin \angle \big( \dot{\tilde{c}} (s), (\partial / \partial t)|_{\tilde{c}(s)} \big) 
\le f\big( t(s) \big).
\]
Since $f'<0$ on $(\rho,\infty)$ and $t(s_0)<t(0)$, this shows $t(s_0)<\rho$.
Thus $\tilde{c}$ intersects the parallel $t= \rho$ 
twice in $\theta^{-1} ((\theta (\tilde{x}), \theta (\tilde{x}) + \pi))$, where we assume that 
$\theta (\tilde{x}) \le \theta (\tilde{y})$. 
However, since $f'(\rho) =0$, the parallel $t = \rho$ is geodesic. 
Therefore (by rotation) $\tilde{x}$ has a cut point in $\theta^{-1}((\theta (\tilde{x}), \theta (\tilde{x}) + \pi))$.
This contradicts the structure of $\Cut (\tilde{x})$ (see Remark \ref{2012_09_19_rem4.1}).
$\qedd$
\end{proof}

\begin{lemma}\label{2012_03_29_lem3.2}
If two points $x, y \in M \setminus \ol{B_{\rho}^+ (p)}$ satisfy $d(p, y) > d(p, x) \gg t_0$,
then
\[
c \big( [0,d(x, y)] \big) \cap \partial B_{t_0}^+ (p) = \emptyset
\]
holds for any minimal geodesic segment $c$ emanating from $x$ to $y$,
where $t_0 > \rho$ is as in the assumption of Theorem~{\rm A}.
\end{lemma}

\begin{proof}
By the assumption (1) of Theorem~A, there is a constant $C > 0$ such that
\begin{equation}\label{2012_03_29_lem3.2_1}
\frac{\diam(\partial B_t^+ (p))}{t^\alpha} < C
\end{equation}
for all $t \gg t_0$.
Suppose that
$c ([0, d(x, y)]) \cap \partial B_{t_0}^+ (p) \not= \emptyset$
for some minimal geodesic segment $c$ emanating from $x$ to $y$. 
Let $S$ be the set of all $s \in (0, d(x, y))$ such that $c(s) \in \partial B_{t_0}^+ (p)$,
and set $s_0:= \sup S$.
Since $d(p, y) > d(p, x)$, there exists $s_1 \in (s_0, d(x, y))$ such that 
$c(s_1) \in \partial B_{t_1}^+(p)$, where $t_1:= d(p, x)$.
Observe from the triangle inequality that
\[
s_1 -s_0 =d\big( c(s_0), c(s_1) \big) \ge d\big( p, c(s_1) \big) -d\big( p, c(s_0) \big)
 =t_1 -t_0.
\]
Since $\diam(\partial B_{t_1}^+ (p)) \ge s_1 > s_1 -s_0 \ge t_1 -t_0$, we obtain
\[
\frac{\diam(\partial B_{t_1}^+ (p))}{t_1^\alpha} > t_1^{1-\alpha} - \frac{t_0}{t_1^\alpha}.
\]
This contradicts (\ref{2012_03_29_lem3.2_1}),
because $t_1 \gg t_0$ and $\alpha<1$.
$\qedd$
\end{proof}

Analogously to \cite{GS}, we define critical points of the distance function
$d_p:=d(p,\cdot)$ as follows.
Recall \eqref{G_p} for the definition of $\cG_p (x)$.

\begin{definition}\label{2012_03_31_def3.3}
We say that a point $x \in M$ is a {\em forward critical point} for $p \in M$
if, for every $w \in T_{x} M \setminus \{0\}$, 
there exists $v \in \cG_p (x)$ such that $g_v (v, w) \le 0$.
\end{definition}

An important consequence of the criticality is that, for any $y \in M$
and any forward triangle $\triangle (\ora{px}, \ora{py})$, we have $\ora{\angle}x \le \pi/2$.
We can prove Gromov's isotopy lemma \cite{G} by a similar arguments to the Riemannian case.

\begin{lemma}\label{2012_03_31_lem3.4}
Given $0 < r_1 < r_2 \le \infty$, 
if $\overline{B_{r_2}^+(p)} \setminus B_{r_1}^+(p)$ has no critical point for $p \in M$, 
then $\overline{B_{r_2}^+(p)} \setminus B_{r_1}^+(p)$ is homeomorphic to 
$\partial B_{r_1}^+(p) \times [r_1, r_2 ]$.
\end{lemma}

Now we are ready to prove Theorem~A.
\medskip

\noindent {\it Proof of Theorem~{\rm A}.}
By virtue of Lemma~\ref{2012_03_31_lem3.4}, it is sufficient to prove that the set of forward 
critical points for $p$ is bounded. 
Suppose that there is a divergent sequence $\{ q_{i} \}_{i \in \N}$
of forward critical points for $p$. 
Then there exist $i_1, i_2 \in \N$ such that 
\[
d(p, q_{i_2}) > d(p, q_{i_1}) \gg t_0 > \rho.
\]
Let $c:[0, a] \lra M$ be a minimal geodesic segment emanating from $q_{i_1}$ to $q_{i_2}$.
Note that $\ora{\angle} (p c(0) c(a)) \le \pi/2$ by the criticality of $q_{i_1}$,
and $c([0, a]) \cap \partial B_{t_0}^+ (p) = \emptyset$ by Lemma~\ref{2012_03_29_lem3.2}.

We first consider the case where $d(p,q_{i_1}) =\min_{s \in [0,a]} d(p,c(s))$.
For sufficiently small $s_1\in (0,a)$,
the forward triangle $\triangle (\ora{pq_{i_1}},\ora{pc(s_1)})$
admits a comparison triangle $\triangle (\tilde{p} \wt{q_{i_1}} \wt{c(s_1)})$ in $\wt{M}$.
Then, by Theorem~\ref{TCT}, we observe that
$\angle\wt{q_{i_1}} \le \ora{\angle} (p c(0) c(a)) \le \pi/2$.
Since
\[ \tilde{d}(\tilde{p},\wt{q_{i_1}}) =d(p, q_{i_1})
 \le d\big( p, c(s_1) \big) =\tilde{d}\big( \tilde{p}, \wt{c(s_1)} \big), \]
this contradicts Lemma~\ref{2012_03_26_lem3.1}.
If $\min_{s \in [0,a]} d(p,c(s)) <d(p,q_{i_1})$,
then we fix $s_0 \in (0,a)$ such that
$d(p, c(s_0)) = \min_{s \in [0,a]} d(p, c(s))$.
By construction, it holds $\ora{\angle}(pc(s_0)c(a)) =\pi/2$
(note that $\ora{\angle}(pc(s_0)c(a)) >\pi/2$ can not happen by Theorem~\ref{TCT}).
Thus we derive a contradiction from the same argument as the first case.
$\qedd$

\section{Proof of Theorem B}\label{sec4}

Let $(M,F,p)$, $f$, $\rho$ and $t_0$ be as in Theorem~B.
Suppose that the cut locus $\Cut (p)$ of $p$ is not empty.
Then, since $M$ is non-compact, $\Cut(p)$ is an unbounded set
(consider a sequence in the open set
$D_p:=\{ v \in U_pM \,|\, \gamma_v((0,\infty)) \cap \Cut(p) \neq \emptyset \}$
whose limit belongs to the complement
$D_p^c=\{ v \in U_pM \,|\, \gamma_v\ \text{is a ray}\}$,
where $U_pM:=T_pM \cap F^{-1}(1)$ and $\gamma_v(t):=\exp_p(tv)$ for $t \ge 0$).
Let $N(p)$ denote the set of all points $x \in M$ 
admitting at least two minimal geodesic segments emanating from $p$ to $x$. 
Note that $N(p)$ is dense in $\Cut (p)$ (see \cite[Proposition 2.6]{TS}). 

Take a divergent sequence $\{ x_{i} \}_{i \in \N} \subset N(p)$, and fix $i_0 \in \N$ such that $d(p, x_{i_0}) > t_0$.
Since $M$ is non-compact and complete, 
there exists a unit speed ray $\sigma:[0,\infty) \lra M$ emanating from $p$.
Take a divergent sequence $\{r_j\}_{j \in \N} \subset (d(p,x_{i_0}),\infty)$ and, 
for each $j$, let $c_j:[0, a_j] \lra M$ be a unit speed minimal geodesic segment emanating 
from $x_{i_0}$ to $\sigma (r_j)$.
By Lemma~\ref{2012_03_29_lem3.2}, 
$c_j([0,a_j]) \cap \partial B_{t_0}(p) = \emptyset$ holds for all $j \in \N$.

Take a subdivision $s_0 := 0 < s_1 < \cdots < s_{k -1} <s_k :=  a_j$ 
of $[0, a_j]$ such that 
$\triangle (\ora{p c_j (s_{l-1})}, \ora{p c_j (s_l)})$ 
admits a comparison triangle 
$\wt{\triangle}^l := \triangle (\tilde{p} \wt{c_j (s_{l-1})} \wt{c_j (s_l)}) \subset \wt{M}$
for each $l = 1,2, \ldots, k$.
Note that, by the reversibility of $F$,
\begin{equation}\label{2012_04_01_thm4.1_2}
\tilde{d}(\wt{c_j (s_{l-1})}, \wt{c_j (s_l)}) = L_{\rm{m}}(c_j |_{[s_{l-1},\,s_l]}) = s_l - s_{l-1}.
\end{equation}
It follows from Theorem~\ref{TCT} that 
\begin{equation}\label{2012_04_01_thm4.1_3}
\ora{\angle} c_j (s_{l-1})  
\ge 
\angle\big( \tilde{p} \wt{c_j (s_{l-1})} \wt{c_j (s_l)} \big), \qquad 
\ola{\angle}c_j (s_l)
\ge 
\angle \big( \tilde{p} \wt{c_j (s_l)} \wt{c_j (s_{l-1})} \big)
\end{equation}
for each $l = 1,2, \ldots, k$. 
Starting from $\wt{\triangle}^1$, we inductively draw a geodesic triangle 
$\wt{\triangle}^{l + 1} \subset \wt{M}$ which is adjacent to $\wt{\triangle}^l$ so as to have a common side
$\tilde{p} \wt{c_j (s_l)}$, where 
$0 \le \theta (\wt{c_j (s_0)}) \le \theta (\wt{c_j (s_1)}) \le \cdots \le \theta (\wt{c_j (s_k)})$.
We observe from the definition of the angles that
$\ola{\angle} c_j (s_l) + \ora{\angle} c_j (s_l) \le \pi$
for each $l = 1,2, \ldots, k-1$. 
Together with \eqref{2012_04_01_thm4.1_3}, we obtain
\begin{equation}\label{2012_04_01_thm4.1_4}
\angle\big( \tilde{p} \wt{c_j (s_l)} \wt{c_j (s_{l-1})} \big) + 
\angle\big( \tilde{p} \wt{c_j (s_l)} \wt{c_j (s_{l+1})} \big) 
\le \pi.
\end{equation}

Let $\wh{\xi}_j:[0, a_j] \lra \wt{M}$ denote the broken geodesic segment 
consisting of minimal geodesic segments from $\wt{c_j (s_{l-1})}$ to $\wt{c_j (s_l)}$,
$l=1,2,\ldots,k$.
We set $\wh{\xi}_j (s) = (t(\wh{\xi}_j (s)), \theta (\wh{\xi}_j (s)))$. 
Then \eqref{2012_04_01_thm4.1_4} gives us the unit speed (not necessarily minimal) geodesic 
$\wt{\eta}_j :[0, b_j] \lra \wt{M}$ emanating from $\wt{c_j (0)}$ to $\wt{c_j (a_j)}$ 
and passing under $\wh{\xi}_j ([0,a_j])$, i.e., 
$\theta (\wt{\eta}_j) \in [\theta (\wt{c_j (0)}), \theta (\wt{c_j (a_j)})]$ on $[0, b_j]$ 
and $t(\wh{\xi}_j (s)) > t(\wt{\eta}_j (b))$ for all $(s,b) \in (0,a_j) \times (0, b_j)$ with 
$\theta (\wh{\xi}_j (s)) = \theta (\wt{\eta}_j (b))$. 
On the one hand, by \eqref{2012_04_01_thm4.1_2}, we have
\[
L(\wt{\eta}_j) \le L(\wh{\xi}_j) = \sum_{l=1}^k \tilde{d}\big( \wt{c_j (s_{l-1})}, \wt{c_j (s_l)} \big)
= s_k - s_0 = a_j,
\]
where $L(\wt{\eta}_j)$ denotes the length of $\wt{\eta}_j$. 
Moreover, the reversibility of $F$ and the triangle inequality show
\begin{equation}\label{2012_04_01_thm4.1_5}
L(\wt{\eta}_j) \le a_j = d\big( x_{i_0}, \sigma(r_j) \big) \le d(p, x_{i_0}) + r_j .
\end{equation}
On the other hand, it follows from \eqref{length-ineq_c} that 
\begin{align*}
L(\wt{\eta}_j) 
&\ge r_j - d(p, x_{i_0}) 
+ \frac{\nu_j^{2}}{2} \int_{d(p, \,x_{i_0})}^{r_j} 
 \frac{1}{f(t)\sqrt{f(t)^2 - \nu_j^{2}}} \,dt\\
& \ge r_j - d(p, x_{i_0}) 
+ \frac{\nu_j^{2}}{2}
\int_{d(p, \,x_{i_0})}^{r_j}  f(t)^{-2} \,dt,
\end{align*}
where $\nu_j$ denotes the Clairaut constant of $\wt{\eta}_j$. 
Together with \eqref{2012_04_01_thm4.1_5}, we find
\[
4 d(p, x_{i_0}) \ge \nu_j^2
\int_{d(p, \,x_{i_0})}^{r_j} f(t)^{-2} \,dt.
\]
Since $f$ is decreasing on $(\rho, \infty)$ because of $f'(\rho) =0$ and $G(\rho) =0$ for a unique 
$\rho \in (0, \infty)$, this implies $\lim_{j \to \infty} \nu_j = 0$. 
Hence we have
\[
\lim_{j \to \infty} \angle\big( \dot{\wt{\eta}}_j(0), (\partial / \partial t)|_{\wt{\eta}_j(0)} \big) =0.
\]
Combining this with
$\angle(\dot{\wt{\eta}}_j(0),(\partial/\partial t)|_{\wt{\eta}_j(0)})
= \pi-\angle(\tilde{p} \wt{c_j(0)} \wt{c_j(s_1)})$
and \eqref{2012_04_01_thm4.1_3}, we obtain
$\lim_{j \to \infty} \ora{\angle} c_j (0)= \pi$.
This is a contradiction, since $c_j(0) = x_{i_0}\in N(p)$.
Hence $\Cut (p)=\emptyset$, so that $M$ is diffeomorphic to $\R^n$. 
The curvature equality follows from the same argument as \cite[Theorem~4.8]{KT3}.
$\qedd$

\medskip

\begin{flushleft}
K.~Kondo, M.~Tanaka\\
Department of Mathematics, Tokai University,\\
Hiratsuka City, Kanagawa Pref. 259-1292, Japan\\
{\small e-mail: {\tt keikondo@keyaki.cc.u-tokai.ac.jp},
{\tt tanaka@tokai-u.jp}}
\bigskip

S.~Ohta\\
Department of Mathematics, Kyoto University,\\
Kyoto 606-8502, Japan\\
{\small e-mail: {\tt sohta@math.kyoto-u.ac.jp}}
\end{flushleft}

\end{document}